\documentclass{elsarticle}

\usepackage{amssymb, latexsym, comment, amsmath, amsthm, upgreek, epsf, epsfig, color, verbatim, caption, tabularx, xfrac, float}
\usepackage{amsfonts}
\usepackage{booktabs}
\usepackage{siunitx}
\usepackage{footnote}
\usepackage{url}

\usepackage{graphicx}              
\usepackage{amsmath}               
\usepackage{amsfonts}              
\usepackage{amsthm}                
\usepackage{amssymb, pgf}
\usepackage{xcolor}
\usepackage{wrapfig}
\usepackage{tikz}
\usetikzlibrary{backgrounds}

\usepackage{array}
\newcolumntype{L}[1]{>{\raggedright\let\newline\\\arraybackslash\hspace{0pt}}m{#1}}
\newcolumntype{C}[1]{>{\centering\let\newline\\\arraybackslash\hspace{0pt}}m{#1}}
\newcolumntype{R}[1]{>{\raggedleft\let\newline\\\arraybackslash\hspace{0pt}}m{#1}}

\makeatletter
\def\ps@pprintTitle{%
\let\@oddhead\@empty
\let\@evenhead\@empty
\def\@oddfoot{\centerline{\thepage}}%
\let\@evenfoot\@oddfoot}
\makeatother

\newtheorem{theorem}{Theorem}[section]
\newtheorem{lemma}[theorem]{Lemma}
\newtheorem{proposition}[theorem]{Proposition}

\newtheorem{definition}[theorem]{Definition}

\newtheorem{conjecture}[theorem]{Conjecture}

\begin{document}
\pagecolor{white}

\begin{frontmatter}
\title{Closed geodesics on doubled polygons}
\author{Ian M Adelstein and Adam YW Fong} 
\address{Department of Mathematics, Yale University \\ New Haven, CT 06520 United States}
\address{Department of Mathematics, Trinity College\\ Hartford, CT 06106 United States}
\begin{abstract} In this paper we study 1/k-geodesics, those closed geodesics that minimize on any subinterval of length $L/k$, where $L$ is the length of the geodesic. We investigate the existence and behavior of these curves on doubled polygons and show that every doubled regular $n$-gon admits a $1/2n$-geodesic. For the doubled regular $p$-gons, with $p$ an odd prime, we conjecture that $k=2p$ is the minimum value for $k$ such that the space admits a $1/k$-geodesic. 
\end{abstract}
\end{frontmatter}

\section{Introduction}

Traders and explorers have long sought shorter paths across our globe. Columbus in the fifteenth century thought it was possible to reach the East by sailing west. Alas, a continent stood in the way, and in the nineteenth century many explorers searched for the elusive Northwest Passage, a sea route connecting the Atlantic and Pacific via the Arctic Ocean. With the advent of air travel more direct routes became possible; planes often follow the shortest path between two points on the globe. In flat Euclidean space (like the $xy$-plane) the shortest path between any two points is a straight line. On a sphere the shortest paths are great circles:~those curves of intersection between the surface of the sphere and a plane containing its center. This is why when you fly between cities in the northern hemisphere your route travels north towards the pole (see Figure~\ref{fig:sphere}).

A geodesic is a locally length minimizing curve; it is the shortest path between any pair of sufficiently close points on the curve. In flat Euclidean space the geodesics are straight lines. We note that these geodesics are not only locally length minimizing, but also globally length minimizing; the straight line is the shortest path between any pair of points on the line, regardless of how close they are. In this paper we study geodesics that fail to minimize globally. As a first example of such a curve consider the geodesic in Figure~\ref{fig:local}. Another important class of geodesics that fail to minimize globally are the closed geodesics, those geodesics that close up on themselves after finite time.

\begin{definition}
We use the symbol $S^1$ to denote the circle. A closed geodesic is a map $\gamma \colon S^1 \rightarrow M $ that is locally length minimizing at every $t \in S^1$.
\end{definition}

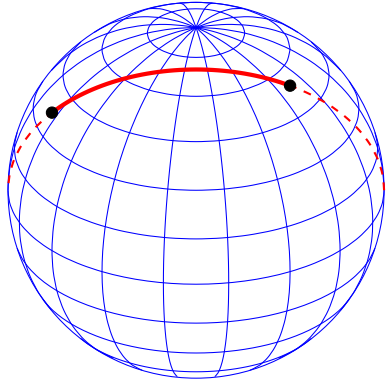
\begin{figure} 
	\begin{center}
		\begin{tikzpicture}[scale=2.5]
			\draw[blue] (0,0) circle (1cm);	
		
			\draw[variable=\x, domain=210:330, smooth, blue]
				 plot ({cos(-3*90/6)*cos(\x)},
				 	{sin(30)*cos(-3*90/6)*sin(\x)
						+cos(30)*sin(-3*(90/6))
					});
			\draw[variable=\x, domain=200:340, smooth, blue]
				 plot ({cos(-2*90/6)*cos(\x)},
				 	{sin(30)*cos(-2*90/6)*sin(\x)
						+cos(30)*sin(-2*(90/6))
					});
			\draw[variable=\x, domain=190:350, smooth, blue]
				 plot ({cos(-1*90/6)*cos(\x)},
				 	{sin(30)*cos(-1*90/6)*sin(\x)
						+cos(30)*sin(-1*(90/6))
					});
			\draw[variable=\x, domain=180:360, smooth, blue]
				 plot ({cos(0*90/6)*cos(\x)},
				 	{sin(30)*cos(0*90/6)*sin(\x)
						+cos(30)*sin(0*(90/6))
					});
			\draw[variable=\x, domain=170:370, smooth, blue]
				 plot ({cos(1*90/6)*cos(\x)},
				 	{sin(30)*cos(1*90/6)*sin(\x)
						+cos(30)*sin(1*(90/6))
					});
			\draw[variable=\x, domain=160:380, smooth, blue]
				 plot ({cos(2*90/6)*cos(\x)},
				 	{sin(30)*cos(2*90/6)*sin(\x)
						+cos(30)*sin(2*(90/6))
					});
			\draw[variable=\x, domain=150:390, smooth, blue]
				 plot ({cos(3*90/6)*cos(\x)},
				 	{sin(30)*cos(3*90/6)*sin(\x)
						+cos(30)*sin(3*(90/6))
					});
			\draw[variable=\x, domain=0:360, smooth, blue]
				 plot ({cos(4*90/6)*cos(\x)},
				 	{sin(30)*cos(4*90/6)*sin(\x)
						+cos(30)*sin(4*(90/6))
					});
			\draw[variable=\x, domain=0:360, smooth, blue]
				 plot ({cos(5*90/6)*cos(\x)},
				 	{sin(30)*cos(5*90/6)*sin(\x)
						+cos(30)*sin(5*(90/6))
					});
		
			\draw[variable=\x, domain=0:180, smooth, blue]
				plot ({cos(\x)*cos(20*0)}, 
					{sin(\x)*cos(30)+sin(30)*(-cos(\x)*sin(20*0))});				
			\draw[variable=\x, domain=-35:145, smooth, blue]
				plot ({cos(\x)*cos(20*1)}, 
					{sin(\x)*cos(30)+sin(30)*(-cos(\x)*sin(20*1))});
			\draw[variable=\x, domain=-55:135, smooth, blue]
				plot ({cos(\x)*cos(20*2)}, 
					{sin(\x)*cos(30)+sin(30)*(-cos(\x)*sin(20*2))});
			\draw[variable=\x, domain=-55:120, smooth, blue]
				plot ({cos(\x)*cos(20*3)}, 
					{sin(\x)*cos(30)+sin(30)*(-cos(\x)*sin(20*3))});
			\draw[variable=\x, domain=-55:120, smooth, blue]
				plot ({cos(\x)*cos(20*4)}, 
					{sin(\x)*cos(30)+sin(30)*(-cos(\x)*sin(20*4))});			
			\draw[variable=\x, domain=-35:145, smooth, blue]
				plot ({cos(\x)*cos(20*8)}, 
					{sin(\x)*cos(30)+sin(30)*(-cos(\x)*sin(20*8))});
			\draw[variable=\x, domain=-55:135, smooth, blue]
				plot ({cos(\x)*cos(20*7)}, 
					{sin(\x)*cos(30)+sin(30)*(-cos(\x)*sin(20*7))});
			\draw[variable=\x, domain=-55:120, smooth, blue]
				plot ({cos(\x)*cos(20*6)}, 
					{sin(\x)*cos(30)+sin(30)*(-cos(\x)*sin(20*6))});
			\draw[variable=\x, domain=-55:120, smooth, blue]
				plot ({cos(\x)*cos(20*5)}, 
					{sin(\x)*cos(30)+sin(30)*(-cos(\x)*sin(20*5))});
			
			\draw[variable=\x, domain=0:180, smooth, thick, red, dashed] 
				plot ({cos(\x)},{sin(40)*sin(\x)});
			\draw[variable=\x, domain=60:140, smooth, ultra thick, red] 
				plot ({cos(\x)},{sin(40)*sin(\x)});
			\draw[fill] ({cos(60)}, {sin(40)*sin(60)}) circle (.3mm);
			\draw[fill] ({cos(140)}, {sin(40)*sin(140)}) circle (.3mm);	
		\end{tikzpicture}
	\end{center}
	
	\caption{Great circle on a sphere showing the shortest path.}
	\label{fig:sphere}
\end{figure}

\begin{figure}
	\begin{center}
		\begin{tikzpicture}[scale=1.6]
			\draw (-1,0) -- (-1,2);
			\draw (1,0) -- (1,2);
			\draw[variable=\x, domain=180:360, smooth]
				plot({cos(\x)},{sin(\x)*cos(70)});
			\draw[variable=\x, domain=0:360, smooth]
				plot({cos(\x)},{sin(\x)*cos(70)+2});
			\draw[variable=\x, domain=1.35:2.5, smooth, dashed]
				plot( {sin(\x*180)},
					{-cos(\x*180)*sin(20)+\x*cos(70)});
			\draw[variable=\x, domain=3.5:4.1, smooth, dashed]
				plot( {sin(\x*180)},
					{-cos(\x*180)*sin(20)+\x*cos(70)});
			\draw[variable=\x, domain=3.5:3.8, smooth, ultra thick, red]
				plot( {sin(\x*180)},
					{-cos(\x*180)*sin(20)+\x*cos(70)});
			\draw[variable=\x, domain=2:2.5, smooth, ultra thick, red]
				plot( {sin(\x*180)},
					{-cos(\x*180)*sin(20)+\x*cos(70)});
			\draw[variable=\x, domain=2.5:3.5, smooth, ultra thick, red, dashed]
				plot( {sin(\x*180)},
					{-cos(\x*180)*sin(20)+\x*cos(70)});
			\draw[fill=black] ( {sin(2*180)},
					{-cos(2*180)*sin(20)+2*cos(70)}) 
					circle (.6mm) node[below right] {$A$};
			\draw[fill=black] ( {sin(3.8*180)},
					{-cos(3.8*180)*sin(20)+3.8*cos(70)}) 
					circle (.6mm) node[below right] {$B$};
		\end{tikzpicture}
	\end{center}
	\caption{Geodesic on the cylinder that is not the shortest path between points $A$ and $B$.}
\label{fig:local}
\end{figure}
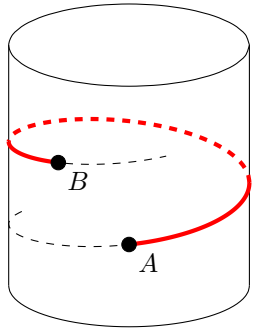

The great circles on the sphere are examples of closed geodesics. Fixing any point on the curve, the great circle is the shortest path to every other point on the circle up to its antipodal point, halfway along the length of the curve. If we traverse past the antipodal point, then a shorter path can be found by traversing the circle in the opposite direction, demonstrating that the great circles are not globally length minimizing. Indeed, every closed geodesic fails to be globally length minimizing, as traversing in the opposite direction always guarantees a shorter path to points beyond the halfway point. 

\begin{figure}
	\begin{center}
		\begin{tikzpicture}[scale=1.2]
			\draw[thick] (0,0) rectangle (6,4);
			\draw[red, thick] (0,0) -- (3,4);
			\draw[red, thick] (3,0) -- (6,4);
			\draw[fill=black] (3,4) circle (.8mm) node[below right] {$q$};
			\draw[fill=black] (3,0) circle (.8mm) node[above left] {$q$};
			\draw[green!60!black, dashed, thick] (1.5,2) -- (4.5,2);
			\draw[blue, dotted, thick, variable=\t, domain=0:1.91] plot (\t+1.5, {2-3/4*(\t)});
			\draw[fill=black] (1.5,2) circle (.8mm) node[above left] {$p$};
			\draw[fill=black] (4.5,2) circle (.8mm) node[below right] {$s$};
		\end{tikzpicture}
	\end{center}
	\caption{Closed geodesic on a flat torus.}
	\label{fig:torus}
\end{figure}
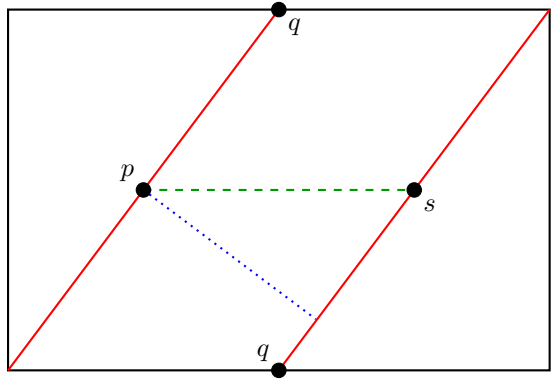

It is not the case that a closed geodesic will always be the shortest path between pairs of points halfway along the curve.  In Figure~\ref{fig:torus} we see an example of a closed geodesic on a flat torus (the red curve) which does not minimize between pairs of points that are half the length apart. Indeed, the green (dashed) curve provides a shorter path between $p$ and $s$. Logically, this poses the question of the largest interval on which a given closed geodesic minimizes. To examine this, Sormani introduced the notion of a $1/k$-geodesic \cite[Definition3.1]{Sor}.

\begin{definition}A $1/k$-geodesic is a constant speed closed geodesic $\gamma \colon S^1 \to M$ which minimizes on all subintervals of length $L/k$, where $L$ is the length of the geodesic and $k \in \mathbb{N}$.
\end{definition}

Note that the great circles on the sphere are $1/2$-geodesics, or half-geodesics. The curve in Figure~\ref{fig:torus} is a $1/4$-geodesic, as it minimizes between all points at length $L/4$ (for example, between the points $p$ and $q$). The curve does not minimize beyond points at length $L/4$, as is evidenced by the blue (dotted) curve between $p$ and a point on the geodesic beyond $q$. See also \cite{Ade}, \cite{ade2}, \cite{wkh}, and \cite{Sor} for more on $1/k$-geodesics. An important first fact about $1/k$-geodesics is that they are as ubiquitous as closed geodesics.

\begin{proposition}[\cite{Sor}, Theorem 3.1] Every closed geodesic is a $1/k$-geodesic for some $k \geq 2$
\end{proposition}

\begin{proof}
 Let $\gamma \colon S^1 \to M$ be a constant speed closed geodesic.  Then by the local length minimization property of $\gamma$ we have for every $t \in S^1=[0,2\pi]$ that there exists an $\epsilon_t > 0$ such that $\gamma$ minimizes on the interval $(t-\epsilon_t, t+\epsilon_t)$. These intervals form an open cover of $S^1$ and by compactness of the circle we can choose a finite subcover. Let $\epsilon$ be the Lebesgue number of the finite subcover, and by the Archimedean property choose $k \geq 2 \pi / \epsilon$. Then $\gamma$ minimizes on all parameter intervals $(t - \pi/k, t+ \pi/k)$ and hence $\gamma$ minimizes on all subintervals of length $L/k$.
\end{proof}

\section{The Over Under-Curve on Doubled Polygons}

We proceed by studying $1/k$-geodesics on doubled regular $n$-gons. We define a doubled regular $n$-gon, denoted $X_n$, to be the metric space obtained by gluing two regular $n$-gons along their common edges. We think of the doubled regular $n$-gons as having a top face and bottom face, so that traversal from one face to the other is possible only by crossing through a point along the shared edges or vertices of the faces. The distance between any two points lying on the same face is the standard Euclidean distance, whereas the distance between two points $x,y \in X_n$ lying on opposite faces is given by $\min_z \{d(x,z) + d(z,y)\}$, where $d$ is the Euclidean distance function on each face and the minimum is taken over all edge points $z \in X_n$.  

We next need to determine the behavior of geodesics on these doubled polygons. On any given face the space is Euclidean and the geodesics are straight lines; if two points are on the same face the straight line path between them is a geodesic. If two points are on opposite faces, a geodesic connecting them must consist of a straight line segment on each face, connected via a shared edge or vertex point. If this geodesic traverses an edge, we can reflect the doubled polygon over this edge, creating a Euclidean space, and conclude that the geodesic on this reflected space must be a straight line. Upon un-reflecting over the edge, we see that the angle of incidence is equal to the angle of reflection, i.e.~that the geodesics billiard around the edges of the doubled polygons, c.f.~\cite{veech}. An application of Heron's solution to the shortest path problem illuminates this billiard behavior. We also have the following lemma.

\begin{lemma}[\cite{Ade}, Lemma 2.1] Geodesics on a doubled regular $n$-gons do not contain vertices as interior points.
\end{lemma}

\begin{proof} By contradiction assume that the geodesic contains a vertex point. Because regular polygons are convex, we can always reflect the doubled polygon over one of the edges adjacent to the vertex (as in the above paragraph) such that the geodesic in the resulting Euclidean space is kinked with an acute angle. Choosing a pair of geodesic points on either side of the vertex, and considering the triangle formed in the resultant Euclidean space from these two points and the vertex, we conclude via the triangle inequality that there exists a shorter path connecting these points. This contradicts the local length minimizing property of the geodesic at the vertex. 
\end{proof}

The closed geodesics on the doubled regular polygons are interesting to study because of their simplicity. Our research is motivated by the following result:

\begin{proposition}[\cite{Ade}, Proposition 2.5] Let $X_n$ be a doubled regular $n$-gon.
   \begin{enumerate}
   \item If $n$ is odd then $X_n$ has no half-geodesics
   \item If $n$ is even then $X_n$ has exactly $\frac{n}{2}$ half-geodesics:~those curves which pass through the center of each face and perpendicularly through parallel edges.
 \end{enumerate}
\end{proposition}

\begin{figure}
	\begin{center}
		\begin{tikzpicture}[scale=1.2]
			\foreach \t in {0,...,3}{
				\draw[red, thick] ({cos(90*\t)}, {sin(90*\t)}) -
					- (0, 0);
				\draw[thick] ({sqrt(2)*cos(45+90*\t)}, {sqrt(2)*sin(45+90*\t)}) -
					- ({sqrt(2)*cos(45+90*(\t+1))}, {sqrt(2)*sin(45+90*(\t+1))});
			}
		\end{tikzpicture}\hfill
		\begin{tikzpicture}[scale=1.2]
			\foreach \t in {0,...,5}{
				\draw[red, thick] ({cos(60*\t)}, {sin(60*\t)}) -
					- (0, 0);
				\draw[thick] ({(1/cos(30))*cos(30+60*\t)}, {(1/cos(30))*sin(30+60*\t)}) -
					- ({(1/cos(30))*cos(30+60*(\t+1))}, {(1/cos(30))*sin(30+60*(\t+1))});
			}
		\end{tikzpicture}\hfill
		\begin{tikzpicture}[scale=1.2]
			\foreach \t in {0,...,7}{
				\draw[red, thick] ({cos(45*\t)}, {sin(45*\t)}) -
					- (0, 0);
				\draw[thick] 
					({(1/cos(22.5))*cos(22.5+45*\t)}, 
						{(1/cos(22.5))*sin(22.5+45*\t)}) -
					- ({(1/cos(22.5))*cos(22.5+45*(\t+1))}, 
						{(1/cos(22.5))*sin(22.5+45*(\t+1))});
			}
		\end{tikzpicture}\hfill
		\begin{tikzpicture}[scale=1.2]
			\foreach \t in {0,1,2}{
				\draw[thick] (\t*.5,0) circle (.5mm);
			}
			\draw[white] (0,-1) circle (.1mm);
			\draw[white] (0,1) circle (.1mm);
		\end{tikzpicture}
	\end{center}
	\caption{The $\frac{n}{2}$ half-geodesics on $X_n$, $n$ even. Note that we only depict one face of the doubled polygon, and that these geodesics are the concatenation of straight line paths on the top and bottom faces.}
\end{figure}
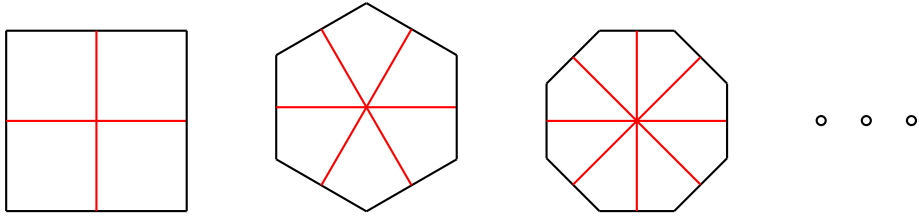

For $n$ odd, the result states that $X_n$ admits no half-geodesics. This naturally leads to the question of the smallest $k \in \mathbb{N}$ such that $X_n$ admits a $1/k$-geodesic. To examine this question we introduce the notion of an over-under curve on $X_n$.

\begin{definition}
Let $\gamma \colon S^1 \rightarrow X_n$ be the closed geodesic on the doubled regular n-gon that passes through the midpoints of adjacent edges of $X_n$. We call $\gamma$ an \emph{over-under} curve between adjacent edges on $X_n$.
\end{definition}

If $\gamma$ is an over-under curve and $\gamma(t_0)$, $\gamma(t_1)$, and $\gamma(t_2)$ are edge points of $X_n$ with the edge containing $\gamma(t_1)$ adjacent to the edges containing $\gamma(t_0)$ and $\gamma(t_2)$ then the following facts are immediate: 
\begin{enumerate}
\item $\gamma \vert_{(t_0,t_1)}$ and $\gamma \vert_{(t_1,t_2)}$ are on opposite faces of $X_n$
\item for every $t \in (t_0,t_1)$ and $s \in (t_1,t_2)$ the minimum path between $\gamma(t)$ and $\gamma(s)$ through the edge containing $\gamma(t_1)$ passes through the point $\gamma(t_1)$
\end{enumerate}

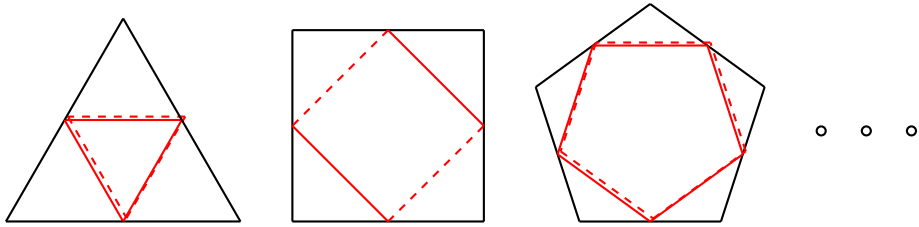
\begin{figure}
	\begin{center}
		\begin{tikzpicture}[scale=.9]
			\foreach \t in {0,...,2}{
				\draw[red, thick, dashed] ({cos(30+120*\t)+.05}, {sin(30+120*\t)+.05}) -
					- ({cos(30+120*(\t+1))+.05}, {sin(30+120*(\t+1))+.05});
			}
			\foreach \t in {0,...,2}{
				\draw[red, thick] ({cos(30+120*\t)}, {sin(30+120*\t)}) -
					- ({cos(30+120*(\t+1))}, {sin(30+120*(\t+1))});
				\draw[thick] ({(1/cos(60))*cos(90+120*\t)}, {1/cos(60))*sin(90+120*\t)}) -
					- ({(1/cos(60))*cos(90+120*(\t+1))}, {(1/cos(60))*sin(90+120*(\t+1))});
			}
		\end{tikzpicture}\hfill
		\begin{tikzpicture}[scale=.9]
			\foreach \t in {0,...,3}{
				\draw[thick] ({(1/cos(60))*cos(45+90*\t)}, {1/cos(60))*sin(45+90*\t)}) -
					- ({(1/cos(60))*cos(45+90*(\t+1))}, {(1/cos(60))*sin(45+90*(\t+1))});
			}
			\foreach \t in {0,2}{
				\draw[red, thick] 
					({(cos(45)/cos(60))*cos(90*\t)}, 
						{(cos(45)/cos(60))*sin(90*\t)}) -
					- ({(cos(45)/cos(60))*cos(90*(\t+1))},
						 {(cos(45)/cos(60))*sin(90*(\t+1))});
				\draw[red, thick, dashed] 
					({(cos(45)/cos(60))*cos(90*(\t+1))}, 
						{(cos(45)/cos(60))*sin(90*(\t+1))}) -
					- ({(cos(45)/cos(60))*cos(90*(\t+2))},
						 {(cos(45)/cos(60))*sin(90*(\t+2))});
			}
		\end{tikzpicture}\hfill
		\begin{tikzpicture}[scale=.8]
			\foreach \t in {0,...,4}{
				\draw[thick] ({2*cos(18+72*\t)}, {2*sin(18+72*\t)}) -
					- ({2*cos(18+72*(\t+1))}, {2*sin(18+72*(\t+1))});
			}
			\foreach \t in {0,...,4}{
				\draw[red, thick] 
					({2*cos(36)*cos(54+72*\t)}, 
						{(2*cos(36)*sin(54+72*\t)}) -
					- ({2*cos(36)*cos(54+72*(\t+1))},
						 {2*cos(36)*sin(54+72*(\t+1))});
				\draw[red, thick, dashed] 
					({2*cos(36)*cos(54+72*(\t+1))+.05}, 
						{2*cos(36)*sin(54+72*(\t+1))+.05}) -
					- ({2*cos(36)*cos(54+72*(\t+2))+.05},
						 {2*cos(36)*sin(54+72*(\t+2))+.05});
			}
		\end{tikzpicture}\hfill
		\begin{tikzpicture}[scale=1.2]
			\foreach \t in {0,1,2}{
				\draw[thick] (\t*.5,0) circle (.5mm);
			}
			\draw[white] (0,-1) circle (.1mm);
			\draw[white] (0,1) circle (.1mm);
		\end{tikzpicture}
	\end{center}
	\caption{ Over-under curves on $X_n$. Note that we now depict as solid the segments of the geodesic on the top face, and as dashed the segments on the bottom face.}
\label{asdf}
\end{figure}

The over-under curves on $X_n$ exhibit distinct behavior depending on the parity of $n$. If $n$ is even, the curves close smoothly after $n$ segments. If $n$ is odd, the curves close after $n$ segments, but not smoothly. The $1^{st}$ and $n^{th}$ segments are on the same face of $X_n$, thus forming a corner when they meet at an edge. The curve needs $2n$ segments before closing smoothly, so that the $1^{st}$ and $2n^{th}$ segments are on opposite faces (see Figure~\ref{asdf}). The following theorem states that the minimizing index of the over-under curves equals the number of segments.

\begin{theorem} 
Let $\gamma \colon S^1 \rightarrow X_n$ be an over-under curve between adjacent edges on a doubled regular $n$-gon. 
\begin{enumerate} 
  \item If $n$ is even then $\gamma$ is a $1/n$-geodesic.
  \item If $n$ is odd then $\gamma$ is a $1/2n$-geodesic.
\end{enumerate}
\end{theorem}

\begin{proof} We prove the theorem for $n$ even and note that the proof of the odd case is equivalent after a reparameterization of the curve. Start by parameterizing $\gamma$ by a circle of length $2\pi$ so that each edge point is given by $p_i = \gamma (2\pi i /n)$. To prove the theorem we show that $\gamma$ is the minimizing path between any pair of points $q_1 = \gamma(t)$ and $q_2 = \gamma (t + 2\pi / n)$. First note that if the $q_j$ are edge points then $\gamma$ is indeed the minimizing path, as $\gamma$ is a straight line path on a single face of $X_n$. Otherwise the $q_j$ are on opposite faces and the segment of $\gamma$ connecting the pair contains an edge point $p_i$. Any shorter path between the $q_j$ must cross an edge distinct from the edge containing $p_i$. It is only necessary to consider paths through the edges containing $p_{i \pm 1}$ as we can easily provide a lower bound of $l(\gamma)/n$ for the length of paths through other edges. Without loss of generality we consider only those paths through the edge containing $p_{i+1}$. 

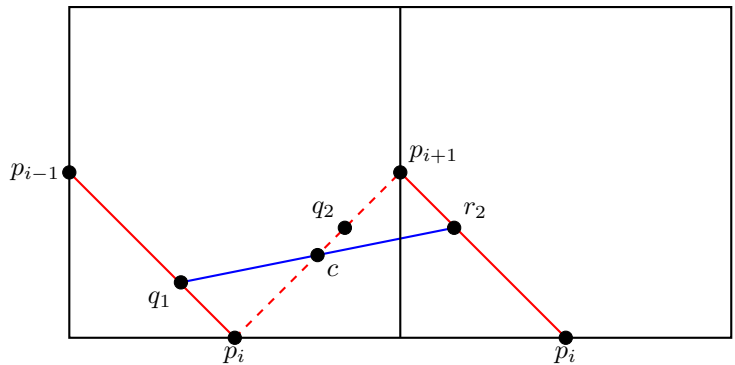
\begin{figure}
	\begin{center}
		\begin{tikzpicture}[scale=1.1]
			\draw[red, thick] (0,2) -- (2,0);
			\draw[red, thick] (4,2) -- (6,0);
			\draw[blue, thick, variable=\t, domain=-1.651:1.651] plot ({\t+3},{1+.2*\t});
			\draw[red, thick, dashed] (2,0) -- (4,2);
			\draw[thick] (0,0) rectangle (8,4);
			\draw[thick] (4,0) -- (4,4);
			\draw[fill=black] (0,2) circle (.8mm) node[left] {$p_{i-1}$};
			\draw[fill=black] (4,2) circle (.8mm) node[above right] {$p_{i+1}$};
			\draw[fill=black] (2,0) circle (.8mm) node[below] {$p_i$};
			\draw[fill=black] (6,0) circle (.8mm) node[below] {$p_i$};
			\draw[fill=black] (3,1) circle (.8mm) node[below right] {$c$};
			\draw[fill=black] ({3-1.651},{1+.2*(-1.651)}) 
				circle (.8mm) node[below left] {$q_1$};
			\draw[fill=black] ({3+1.651},{1+.2*(1.651)}) 
				circle (.8mm) node[above right] {$r_2$};
			\draw[fill=black] ({3+.2*1.651},{1+.2*(1.651)}) 
				circle (.8mm) node[above left] {$q_2$};
		\end{tikzpicture}
	\end{center}
	\caption{The over-under curve on $X_4$.}  
  \label{sq_fig}
\end{figure}

By reflecting the doubled polygon over the edge containing $p_{i+1}$ and considering the top and bottom faces as part of the same plane we are able to complete the proof in the Euclidean setting. Assume $q_1$ is on the top face and let $r_2$ denote the reflection of $q_2$ through the edge containing $p_{i+1}$ (see Figure~\ref{sq_fig}). We show that the straight line path between $q_1$ and $r_2$ has length at least $l(\gamma)/n$. Let $c$ be the point of intersection between the line segments $\overline{q_1 r_2}$ and $\overline{p_i p_{i+1}}$. Consider the pair of triangles $\triangle q_1 c p_i$ and $\triangle r_2 c p_{i+1}$. By construction we have that the sides opposite $\angle c$ in each triangle have equal length so that applying law of sines to both triangles yields 

\[\frac{\sin(\angle q_1)}{Q_1} = \frac{\sin(\angle p_i)}{P_i} = \frac{\sin(\angle c)}{C} =  \frac{\sin(\angle r_2)}{R_2} = \frac{ \sin(\angle p_{i+1})}{P_{i+1}}\]
where we have used a capital letter to denote the length of the side opposite its angle. We note that $\angle p_i = \pi - \angle p_{i+1}$ so that $\angle r_2 = \angle p_i - \angle c$ and
\[\frac{\sin(\pi - \angle c - \angle p_i)}{Q_1} = \frac{\sin(\angle p_i)}{P_i} = \frac{\sin(\angle p_i - \angle c)}{R_2} = \frac{\sin(\pi - \angle p_i)}{p_{i+1}} \]
Via the trigonometric identity $\sin(\pi-x) = \sin(x)$ we have that $P_i =  P_{i+1}$ and
\begin{align*}
\frac {Q_1 + R_2}{2 P_i} &= \frac{\sin(\angle p_i - \angle c) + \sin(\angle p_i + \angle c)}{2 \sin(\angle p_i)} = \frac {2\sin(\angle p_i)\cos(\angle c)}{2 \sin(\angle p_i)} \\ 
&= \cos(\angle c) \leq 1
\end{align*}
We have therefore shown that $ 2 P_i = P_i + P_{i+1} \geq Q_1 + R_2 = l(\gamma)/n$ and conclude that $\gamma$ minimizes on all subintervals of length $ l(\gamma)/n$. 
\end{proof}

\section{Bounding the minimizing index}

We have shown for $n$ odd that $X_n$ admits a $1/2n$-geodesic by explicitly constructing such curves. We now consider whether these curves realize the optimal minimizing property on $X_n$, i.e.~if $k=2n$ is the smallest $k \in \mathbb{N}$ for which $X_n$ ($n$ odd) admits a $1/k$-geodesic. To quantify this notion Sormani introduced the minimizing index.
  
\begin{definition}[\cite{Sor}, Definition 3.3] The minimizing index of a metric space $M$, denoted minind(M),  is the smallest $k \in \mathbb{N}$ such that the metric space admits a $1/k$-geodesic.
\end{definition}

For $n$ odd the results of the previous section give an upper bound of $2n$ on $minind(X_n)$. Furthermore, we have seen that such $X_n$ do not admit half-geodesics and consequently that $2 < minind(X_n) \leq 2n$. A natural question is whether we can sharpen this bound on the minimizing index of $X_n$. Given a doubled prime-gon it is compelling to believe that its minimizing index is 2p.

\begin{conjecture}\label{conj}
  If p is an odd prime, then $minind(X_p) = 2p$.
\end{conjecture}

Observe here that the primality of $p$ is necessary, since if we have that $n=kp$ with $k\geq2$, we can construct a $1/2p$-geodesic by creating an over under curve between the midpoints of every $k$th edge of $X_n$. Evidence towards this conjecture begins with the following:

\begin{proposition}\label{three}
The conjecture is true for the case $p=3$, i.e.~the minimizing index of the doubled regular triangle is six. 
\end{proposition}

\begin{proof}
We first define the \emph{period} of a closed geodesic on a doubled polygon to be its total number of segments. As these geodesics must close smoothly, we have that the period is always even. Also note because a geodesic on a doubled polygon will never minimize on an open segment that contains multiple edge points, that the period provides a lower bound on the minimizing index of a geodesic (the smallest $k \in \mathbb{N}$ such that it is a 1/k-geodesic).

We have therefore reduced the problem to showing that those closed geodesics with period less than six have minimizing index at least six. We have already established that $X_3$ does not admit half-geodesics, and that the period must be even, so we need only consider those closed geodesics with period four. Such curves can be classified:~they must leave an edge with angle $\pi/6$, traverse an adjacent edge perpendicularly, return to the starting edge (at the same point, but not with the same velocity), traverse the remaining edge perpendicularly, and return to the starting point to close up smoothly (see Figure~\ref{triangle2}). 

\begin{figure}
	\begin{center}
		\begin{tikzpicture}[scale=3.4]
			\foreach \t in {0,...,2}{
				\draw[thick] ({cos(90+120*\t)}, {sin(90+120*\t)}) -
					- ({cos(90+120*(\t+1))}, {sin(90+120*(\t+1))});
			}
			\draw[thick, variable=\t, domain=0:.65] plot ({\t},{-sin(30)+(tan(30)*\t)});
			\draw[thick, variable=\t, domain=-.65:0] plot ({\t},{-sin(30)-(tan(30)*\t)});
			\draw[thick, dashed] 
				({cos(30)/2},{-sin(30)+tan(30)*cos(30)/2}) -
					- ({cos(30)/2},{-sin(30)});
			\draw[thick, dashed] 
				({cos(30)/2},{-sin(30)+tan(30)*cos(30)/2}) -
					- ({cos(330)},{sin(330)});
			\draw[fill=black] 
				({cos(210)+tan(60)*cos(60)/4},{sin(210)+tan(60)*sin(60)/4}) 
					circle (.2mm) node[above left] {$P_1$};
			\draw[fill=black] 
				(0,{-sin(30)}) circle (.2mm) node[below] {$P_2$};
			\draw[fill=black] 
				({cos(330)-tan(60)*cos(60)/4},{sin(330)+tan(60)*sin(60)/4}) 
					circle (.2mm) node[above right] {$P_3$};
			\draw[fill=black] 
				({cos(30)/2},{-sin(30)}) circle (.2mm) node[below] {$R$};
			\draw[fill=black] 
				({cos(30)/2},{-sin(30)+tan(30)*cos(30)/2}) 
					circle (.2mm) node[above left] {$Q$};
			\draw[fill=black] 
				({cos(330)},{sin(330)}) circle (.2mm) node[below right] {$V$};
		\end{tikzpicture}
	\end{center}
	\caption{Closed geodesic on $X_3$ with period four and minimizing index six.}
  \label{triangle2}
\end{figure}
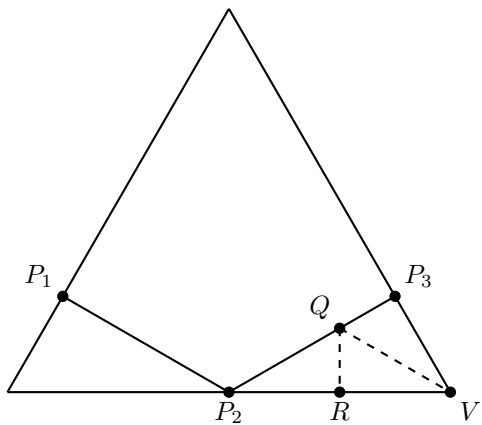

It remains to show that any period four geodesic on $X_3$ has minimizing index at least six. We first show that the period four geodesic from Figure~\ref{triangle2} has minimizing index at least six. In this figure $QV$ is the bisector of angle $V$ and $QR$ is perpendicular to $VP_2$. Using properties of similar triangles we have that $|QR|=|QP_3|=|P_2P_3|/3 = L/12$. This demonstrates that there exist two equal length paths between $Q$ and its corresponding point on the bottom face:~one along our geodesic through $P_3$, and another through $R$. The geodesic therefore can not minimize beyond this segment of length $L/6$, and we conclude that the minimizing index must be at least six. For a period four geodesic on $X_3$ that does not contain the midpoint of an edge, a similar argument shows that the minimizing index must be strictly greater than six. 
\end{proof}

Please note that Proposition~\ref{three} did not appear in the original version of this paper. The proof was sketched by the undergraduate research group \cite{aahs} and independently by one of the referees (who also produced Figure~\ref{triangle2}). The original paper had an argument equivalent to the last paragraph of the proof showing that the minimizing index of the geodesic from Figure~\ref{triangle2} is at least six, but did not classify all period four geodesics, and therefore did not determine the minimizing index of $X_3$. 

It is reasonable to believe that a similar argument could be used to show that $minind(X_5) = 10$. It need only be shown that closed geodesics of period four, six, or eight have minimizing index at least ten. One quickly realizes that this direction of reasoning will prove untenable for resolving the conjecture; as $p$ grows it becomes prohibitively difficult to complete such an analysis. As a partial solution to the conjecture we present the following:

\begin{theorem}[\cite{aahs}, Theorem 2] 
For $p$ prime, as $p \to \infty$, the minimizing index of $X_p$ grows without bound. 
\end{theorem}

This theorem was proved after the completion of this paper by a subsequent undergraduate research group \cite{aahs} working again with the first named author. The proof involves a careful study of the closed geodesics on  doubled polygons, developing new techniques to study their minimizing properties. To the best of our knowledge Conjecture~\ref{conj} remains open, and we invite the reader to pursue their own investigations.

\section{Acknowledgements} The authors would like to thank the Faculty Research Committee at Trinity College for funding the second named author's on-campus research with the first named author through the Student Research Program. We also acknowledge the wonderful work of Brett C.~Smith who recreated all the figures in this paper for publication.


\end{document}